\documentclass[11pt]{amsart}
\usepackage{latexsym,amscd,amssymb, graphicx, color, amsthm, bm, amsmath, tikz}  
\usepackage{young}
\usepackage[margin=1in]{geometry}
\usepackage{hyperref}
\usepackage[all,cmtip]{xy}
\usepackage{mathtools}

\numberwithin{equation}{section}

\newtheorem{theorem}{Theorem}[section]
\newtheorem{proposition}[theorem]{Proposition}
\newtheorem{corollary}[theorem]{Corollary}

\newtheorem{problem}[theorem]{Problem}

\newtheorem{remark}[theorem]{Remark}

\theoremstyle{definition}

\newcommand{\ch}{{\mathrm {ch}}}
\newcommand{\maj}{{\mathrm {maj}}}

\newcommand{\sign}{{\mathrm {sign}}}

\newcommand{\SSYT}{{\mathrm {SSYT}}}

\newcommand{\grFrob}{{\mathrm {grFrob}}}
\newcommand{\des}{{\mathrm {des}}}

\newcommand{\SYT}{{\mathrm {SYT}}}

\newcommand{\Frob}{{\mathrm {Frob}}}

\newcommand{\shape}{{\mathrm {shape}}}

\newcommand{\symm}{{\mathfrak{S}}}

\newcommand{\rank}{{\mathrm{rank}}}

\newcommand{\CC}{{\mathbb {C}}}

\newcommand{\ZZ}{{\mathbb {Z}}}
\newcommand{\PP}{{\mathbb{P}}}
\newcommand{\RR}{{\mathbb{R}}}

\newcommand{\FFF}{{\mathcal{F}}}
\newcommand{\SSS}{{\mathbb{S}}}

\newcommand{\EEE}{{\mathcal{E}}}

\newcommand{\CCC}{{\mathcal{C}}}

\newcommand{\BBB}{{\mathcal{B}}}

\newcommand{\ee}{{\mathbf {e}}}

\begin{document}

\title[Boolean product polynomials]
{Boolean product polynomials, Schur positivity,\\ and Chern plethysm}
\date{\today}

\author{Sara C. Billey}
\address
{Department of Mathematics \newline \indent
University of Washington \newline \indent
Seattle, WA, 98195-4350}
\email{billey@math.washington.edu}

\author{Brendon Rhoades}
\address
{Department of Mathematics \newline \indent
University of California, San Diego \newline \indent
La Jolla, CA, 92093-0112, USA}
\email{bprhoades@math.ucsd.edu}

\author{Vasu Tewari}
\address
{Department of Mathematics \newline \indent
University of Pennsylvania \newline \indent
Philadelphia, PA, 19104-6395, USA}
\email{vvtewari@math.upenn.edu}

\begin{abstract}
  Let $1\leq k \leq n$ and let $X_n = (x_1, \dots, x_n)$ be a list of
  $n$ variables.  The {\em Boolean product polynomial} $B_{n,k}(X_n)$
  is the product of the linear forms $\sum_{i \in S} x_i$ where $S$
  ranges over all $k$-element subsets of $\{1, 2, \dots, n\}$.  We
  prove that Boolean product polynomials are Schur positive.  We do
  this via a new method of proving Schur positivity using vector
  bundles and a symmetric function operation we call {\em Chern
    plethysm}.  This gives a geometric method for producing a vast
  array of Schur positive polynomials whose Schur positivity lacks (at
  present) a combinatorial or representation theoretic proof.  We
  relate the polynomials $B_{n,k}(X_n)$ for certain $k$ to
  other combinatorial objects including derangements, positroids,
  alternating sign matrices, and reverse flagged fillings of a
  partition shape.  We also  relate $B_{n,n-1}(X_n)$ to  
  a bigraded action of the symmetric group $\symm_n$
  on a divergence free quotient of superspace.   

\end{abstract}

\maketitle

\section{Introduction}
\label{Introduction}

The symmetric group $\symm_n$ of permutations of $[n] \coloneqq \{1, 2, \dots, n\}$ acts on the polynomial
ring $\CC[X_n] \coloneqq \CC[x_1, \dots, x_n]$ by variable permutation.
Elements of the invariant subring
\begin{equation}
\CC[X_n]^{\symm_n} \coloneqq 
\{ f(X_n) \in \CC[X_n] \,:\, w.f(X_n) = f(X_n) \text{ for all $w \in \symm_n$ } \}
\end{equation}
are called {\em symmetric polynomials}.

Symmetric polynomials are typically defined using {\bf sums of products} of the variables
$x_1, \dots, x_n$.   Examples include
the {\em power sum}, the 
{\em elementary symmetric polynomial}, and the 
{\em homogeneous symmetric polynomial} which are (respectively)
\begin{equation}
p_k(X_n) = x_1^k + \cdots + x_n^k, \quad 
e_k(X_n) = \sum_{1 \leq i_1 < \cdots < i_k \leq n} x_{i_1} \cdots x_{i_k}, \quad
h_k(X_n) = \sum_{1 \leq i_1 \leq \cdots \leq i_k \leq n} x_{i_1} \cdots x_{i_k}.
\end{equation}
Given a partition $\lambda = (\lambda_1 \geq  \dots \geq  \lambda_k > 0)$ with $k \leq n$
parts, we  have the {\em monomial symmetric polynomial}
\begin{equation}
m_{\lambda}(X_n) = 
\sum_{\substack{\text{$i_1, \dots, i_k$ distinct}}}
x_{i_1}^{\lambda_1} \cdots x_{i_k}^{\lambda_k},
\end{equation}
as well as the {\em Schur polynomial}
 $s_{\lambda}(X_n)$ whose definition is recalled in Section~\ref{Background}.
 
 Among these symmetric polynomials, the Schur polynomials are the most important.
 The set of Schur polynomials $s_{\lambda}(X_n)$ where $\lambda$ has at most $n$ parts
 forms a $\CC$-basis of
$\CC[X_n]^{\symm_n}$.
A symmetric polynomial $f(X_n)$ is {\em Schur positive} if its expansion into the Schur
basis has nonnegative integer coefficients.
Schur positive polynomials admit representation theoretic interpretations involving  
general linear and symmetric groups as well as geometric interpretations involving  cohomology rings of
Grassmannians.
A central problem in the theory of symmetric polynomials is to decide whether a given symmetric polynomial
$f(X_n)$ is Schur positive.

In addition to the sums of products described above, one can also define symmetric polynomials
using {\bf products of sums}.  For $1 \leq k \leq n$, we define the {\em Boolean product polynomial}
\begin{equation}
B_{n,k}(X_n) \coloneqq \prod_{1 \leq i_1 < \cdots < i_k \leq n}
(x_{i_1} + \cdots + x_{i_k}).
\end{equation}
For example, when $n = 4$ and $k = 2$, we have
\begin{equation*}
  B_{4,2}(X_4) = (x_1 + x_2)(x_1 + x_3)(x_1 + x_4)(x_2 + x_3)(x_2 +
  x_4)(x_3 + x_4).  
\end{equation*}
One can check $B_{n,1}=x_1x_2\cdots x_n=s_{(1^n)}(X_n)$ and
$B_{n,2}(X_n)=s_{(n-1,n-2,\ldots,1)}$ for $n\geq 2$.
We also define a {\em `total' Boolean product polynomial} $B_n(X_n)$ to be the product
of the $B_{n,k}$'s, 
\begin{equation}
B_n(X_n) \coloneqq \prod_{k = 1}^n B_{n,k}(X_n).
\end{equation}

Lou Billera provided our original inspiration for studying $B_n(X_n)$ at a BIRS workshop
in 2015.
His study of the  Boolean product polynomials was partially motivated
by the study of the {\em resonance arrangement}.
This is the  hyperplane arrangement in $\RR^n$ with hyperplanes 
$\sum_{i \in S} x_i = 0$,
where $S$ ranges over all nonempty subsets of $[n]$. 
The polynomial $B_n(X_n)$ is the defining polynomial of this arrangement.
The resonance arrangement is related to double Hurwitz numbers \cite{CJM},
quantum field theory \cite{Evans}, and certain preference rankings in psychology and economics \cite{KTT}.
Enumerating the regions of the resonance arrangement is an open problem.
In Section~\ref{Open} we present further motivation for  Boolean product polynomials.


In this paper we prove that $B_{n,k}(X_n)$ and $B_n(X_n)$ are Schur
positive (Theorem~\ref{boolean-schur-positivity}).  These results were
first announced in \cite{BBT} and presented at FPSAC 2018.  The proof
relies on the geometry of vector bundles and involves an operation on
symmetric functions and Chern roots which we call {\em Chern
  plethysm}.  Chern plethysm behaves in some ways like classical
plethsym of symmetric functions, but it is clearly a different
operation.  Our Schur positivity results follow from earlier results
of Pragacz \cite{PAlain} and Fulton-Lazarsfeld \cite{FL} on numerical
positivity in vector bundles over smooth varieties.  This method
provides a vast array of Schur positive polynomials coming from
products of sums, the polynomials $B_{n,k}(X_n)$ and $B_n(X_n)$ among
them.

There is a great deal of combinatorial and representation theoretic
machinery available for understanding the Schur positivity of sums of
products.  Schur positive products of sums are much less understood.
Despite their innocuous definitions, there is no known combinatorial
proof of the Schur positivity of $B_{n,k}(X_n)$ or $B_n(X_n)$, nor is
there a realization of these polynomials as the Weyl character of an
explicit polynomial representation of $GL_n$ for all $k,n$.  It is the
hope of the authors that this paper will motivate further study into
Schur positive products of sums.

Toward developing combinatorial interpretations and related
representation theory for $B_{n,k}(X_n)$, we study the special cases
$B_{n,2}(X_n)$ and $B_{n,n-1}(X_n)$ in more detail.  The polynomials $B_{n,2}(X_n)$
are the highest homogeneous component of certain products famously
studied by Alain Lascoux \cite{Lascoux}, namely
\[\prod_{1\leq i< j\leq n} (1+ x_i +x_j).\]
He showed these polynomials are Schur positive using vector bundles,
inspiring the work of Pragacz. It is nontrivial to show the
coefficients in his expansion are nonnegative integers.  Lascoux's
work was also the motivation for the highly influential work of
Gessel-Viennot on lattice paths and binomial determinants
\cite{GV}. In Theorem~\ref{lascoux-schur-expansion}, we give the first
purely combinatorial interpretation for all of the Schur expansion
coefficients in Lascoux's product.  Surprisingly, the sum of the
coefficients in the Schur expansion of this formula, is equal to the
number of alternating sign matrices of size $n$ or equivalently the
number of totally symmetric self-complementary plane partitions of
$2n$ (Corollary~\ref{cor:total_good_fillings}).

In Section~\ref{Analogue}, we introduce a $q$-analog of $B_{n,n-1}$.
At $q=0$, this polynomial is the Frobenius characteristic of the
regular representation of $\symm_n$.  At $q=-1$, we get back the
Boolean product polynomial.  By work of D\'esarm\'enien-Wachs and
Reiner-Webb, we have a combinatorial interpretation of the Schur
expansion $B_{n,n-1}$.  Furthermore, $B_{n,n-1}$ is the character of a
direct sum of $\mathrm{Lie}_\lambda$ representations which has a basis
given by derangements in $\symm_n$.  At $q=1$, this family of
symmetric functions is related to an $\symm_n$-action on positroids.
We show the $q$-analog of $B_{n,n-1}$ is the graded Frobenius
characteristic of a bigraded $\symm_n$-action on a divergence free
quotient of superspace related to the classical coinvariant algebras,
which we call $R_n$.  Following the work of
Haglund-Rhoades-Shimozono\cite{HRS}, we extend our construction to the
context of ordered set partitions and beyond.

The remainder of the paper is structured as follows.  In {\bf
  Section~\ref{Background}}, we review the combinatorics and
representation theory of Schur polynomials.  In {\bf
  Section~\ref{Chern}}, we introduce Chern plethysm and explain the
relevance of the work of Pragacz to Schur positivity.  In {\bf
  Section~\ref{sec:bn2}}, we introduce the reverse flagged fillings in
relation to Lascoux's formula in our study of Boolean product
polynomials of the form $B_{n,2}$.  In {\bf Section~\ref{Analogue}},
we connect $B_{n,k}$ to combinatorics and representation theory in the
special case $k = n -1$.  In particular, we introduce a
$q$-deformation of $B_{n,n-1}$ and relate it to a quotient of
superspace.  We close in {\bf Section~\ref{Open}} with some open
problems.

\section{Background}
\label{Background}

We provide a brief introduction to the background and notation we are
assuming in this paper.  Further details on Schur polynomials and the
representation theory of $\symm_n$ and $GL_n$ can be found in 
\cite{Fulton-Harris}.

\subsection{Partitions, tableaux, and Schur polynomials}
A {\em partition of $n$} is a weakly decreasing sequence
$\lambda = (\lambda_1 \geq \cdots \geq \lambda_k)$ of positive
integers such that $\lambda_1 + \cdots + \lambda_k = n$. We write
$\lambda \vdash n$ or $|\lambda| = n$ to mean that $\lambda$ is a
partition of $n$.  We also write $\ell(\lambda) = k$ for the number of
parts of $\lambda$.  The {\em Ferrers diagram} of a partition
$\lambda$ consists of $\lambda_i$ left justified boxes in row $i$.
The Ferrers diagram of $(3,3,1) \vdash 7$ is shown on the left in
Figure~\ref{fig:1}.  We identify partitions with their Ferrers diagrams
throughout.

\begin{figure}
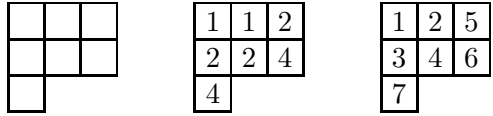

\begin{equation*}
\begin{young}
 & & \cr
 & & \cr
 \cr
\end{young} \hspace{0.4in}
\begin{young}
 1& 1 & 2 \cr
 2& 2 & 4 \cr
 4 \cr
\end{young} \hspace{0.4in}
\begin{young}
 1& 2 & 5 \cr
 3 & 4 & 6  \cr
 7 \cr
\end{young} 
\end{equation*}
\caption{The Ferrers diagram of $(3,3,1)$ along with a semistandard and
  a standard Young tableau of that shape. \label{fig:1}}
\end{figure}

If $\lambda$ is a partition, a {\em semistandard tableau} $T$ of shape
$\lambda$ is a filling $T: \lambda \rightarrow \ZZ_{> 0}$ of the
boxes of $\lambda$ with positive integers such that the entries
increase weakly across rows and strictly down columns. A semistandard
tableau of shape $(3,3,1)$ is shown in the middle of
Figure~\ref{fig:1}.  Let $\SSYT(\lambda, \leq n)$ be the family of all
semistandard tableaux of shape $\lambda$ with entries $\leq n$.  A
semistandard tableau is {\em standard} if its entries are
$1, 2, \dots, |\lambda|$.  A standard tableau, or standard Young
tableau, is shown on the right in Figure~\ref{fig:1}.

Given a semistandard tableau $T$, define a monomial $x^T \coloneqq x_1^{m_1(T)} x_2^{m_2(T)} \cdots $, where
$m_i(T)$ is the multiplicity of $i$ as an entry in $T$.
In the above example, we have $x^T = x_1^2 x_2^3 x_4^2$.  The {\em Schur polynomial}
$s_{\lambda}(X_n)$ is the corresponding generating function
\begin{equation}
s_{\lambda}(X_n) \coloneqq \sum_{T \in \SSYT(\lambda, \leq n)} x^T.
\end{equation}
Observe that $s_{\lambda}(X_n) = 0$ whenever $\ell(\lambda) > n$.

An alternative formula for the Schur polynomial can be given as a
ratio of determinants as follows.  Let
$\Delta_n \coloneqq \prod_{1 \leq i < j \leq n} (x_i - x_j) =\sum_{w
  \in \symm_n} \sign(w) \cdot (w.x_1^{n-1}x_2^{n-2}\cdots x_n^{0})$
  be the Vandermonde determinant.  Given any polynomial
  $f \in \CC[X_n]$, define a symmetric polynomial $A_n(f)$ by
\begin{equation}\label{eq:anti.symm} A_n(f) \coloneqq \frac{1}{\Delta_n} \sum_{w \in
\symm_n} \sign(w) \cdot (w.f).
\end{equation}
Let $\mu = (\mu_1 \geq \cdots \geq \mu_n \geq 0)$ be a partition with
$\leq n$ parts.  The Schur polynomial $s_{\mu}(X_n)$ can also be
obtained applying $A_n$ to the monomial
$x_1^{\mu_1 + n - 1} x_2^{\mu_2 + n - 2} \cdots x_n^{\mu_n}$:
\begin{equation} s_{\mu}(X_n) = A_n(x_1^{\mu_1 + n - 1} x_2^{\mu_2 + n
- 2} \cdots x_n^{\mu_n}).
\end{equation}

\subsection{$GL_n$-modules and Weyl characters}
Let $GL_n$ be the group of invertible $n \times n$ complex matrices, and
let $W$ be a finite-dimensional representation of $GL_n$ with underlying group homomorphism
$\rho: GL_n \rightarrow GL(W)$.
The {\em Weyl character} of $\rho$ is the function
$\ch_{\rho}: (\CC^{\times})^n \rightarrow \CC$ defined by
\begin{equation}
\ch_{\rho}(x_1, \dots, x_n) \coloneqq \mathrm{trace}( \rho(\mathrm{diag}(x_1, \dots, x_n)) )
\end{equation}
which sends an $n$-tuple $(x_1, \dots, x_n)$ of nonzero complex numbers to the trace of the 
diagonal matrix $\mathrm{diag}(x_1, \dots, x_n) \in GL_n$ as an operator on $W$.
The function $\ch_{\rho}$ satisfies 
$\ch_{\rho}(x_1, \dots, x_n) = \ch_{\rho}(x_{w(1)}, \dots, x_{w(n)})$ for any $w \in \symm_n$.

A representation $\rho: GL_n \rightarrow GL(W)$ is {\em polynomial} if $W$ is finite-dimensional and there 
exists a basis $\BBB$ of $W$ such that the entries of the matrix $[\rho(g)]_{\BBB}$ representing $\rho(g)$
are polynomial functions of the entries of $g \in GL_n$.  
This property is independent of the choice of basis $\BBB$.
In this case, the Weyl character
$\ch_{\rho} \in \CC[X_n]^{\symm_n}$ is a symmetric polynomial.
We will only consider polynomial representations in this paper.  

Let $\lambda \vdash d$ and let $T$ be a standard Young tableau with $d$ boxes. Let
$R_T, C_T \subseteq \symm_d$ be the subgroups of permutations in $\symm_d$ which stabilize the rows
and columns of $T$, respectively.  
For the standard tableau of shape $(3,3,1)$ shown in
Figure~\ref{fig:1}, we have
$R_T = \symm_{\{1,2,5\}} \times \symm_{\{3,4,6\}} \times \symm_{\{7\}}$
and
$C_T = \symm_{\{1,3,7\}} \times \symm_{\{2,4\}} \times \symm_{\{5,6\}}$.
The {\em Young idempotent} $\varepsilon_{\lambda} \in \CC[\symm_d]$ is the group algebra 
element
\begin{equation}
\varepsilon_{\lambda} \coloneqq \sum_{w \in R_T} \sum_{u \in C_T} \sign(u) \cdot uw \in \CC[\symm_d].
\end{equation}
Strictly speaking, the group algebra element $\varepsilon_{\lambda}$ depends on the standard tableau $T$,
but this dependence is only up to conjugacy by an element of $\symm_d$ and will be ignored.

Let $V = \CC^n$ be the standard $n$-dimensional complex vector space.
The symmetric group $\symm_d$ acts on the $d$-fold tensor product $V \otimes \cdots \otimes V$ 
{\em on the right} by permuting tensor factors:
\begin{equation}
(v_1 \otimes \cdots \otimes v_d).w \coloneqq v_{w^{-1}(1)} \otimes \cdots \otimes v_{w^{-1}(d)}, \quad v_i \in V, w \in \symm_d.
\end{equation}
By linear extension we have an action of the group algebra $\CC[\symm_d]$ on $V \otimes \cdots \otimes V$.
If $\lambda$ is a partition, 
the {\em Schur functor} $\SSS^{\lambda}(\cdot)$ attached to $\lambda$ is defined by
\begin{equation}
\SSS^{\lambda}(V) \coloneqq (V \otimes \cdots \otimes V) \varepsilon_{\lambda}.
\end{equation}
We have $\SSS^{\lambda}(V) = 0$ whenever $\ell(\lambda) > \dim(V)$.
Two special cases are of interest. If $\lambda = (d)$ is a single row then
$\SSS^{(d)}(V) = Sym^d V$ is the $d^{th}$ symmetric power. If $\lambda = (1^d)$ is a single 
column then $\SSS^{(1^d)}(V) = \wedge^d V$ is the $d^{th}$ exterior power.

The group $GL_n = GL(V)$ acts on $V \otimes \cdots \otimes V$ {\em on the left} by the diagonal action
\begin{equation}
g.(v_1 \otimes \cdots \otimes v_d) \coloneqq (g.v_1) \otimes \cdots \otimes (g.v_d), \quad
v_i \in V,\ g \in GL_n.
\end{equation}
This commutes with the action of $\CC[\symm_d]$ and so turns $\SSS^{\lambda}(V)$ into a $GL_n$-module.
We quote the following bromides of $GL_n$-representation theory.
\begin{enumerate}
\item  If $\ell(\lambda) \leq n$,
the module $\SSS^{\lambda}(V)$ is an irreducible polynomial representation of $GL_n$ with Weyl
character given by the Schur polynomial $s_{\lambda}(X_n)$.
\item The modules $\SSS^{\lambda}(V)$ for $\ell(\lambda) \leq n$ form a complete list of the nonisomorphic
irreducible
polynomial representations of $GL_n$.
\item Any polynomial $GL_n$-representation may be expressed
uniquely as a direct sum of the modules $\SSS^{\lambda}(V)$.
\end{enumerate}
A symmetric polynomial $f(X_n) \in \CC[X_n]^{\symm_n}$ is therefore Schur positive if and only if
it is the Weyl character of a polynomial representation of $GL_n$.

\subsection{$\symm_n$-modules and Frobenius image}
Let $X = (x_1, x_2, \dots )$ be an infinite list of variables.
For $d > 0$, the {\em power sum symmetric function} is 
$p_d(X) = x_1^d + x_2^d + \cdots$; this is an element of the ring $\CC[[X]]$ of formal power series in $X$.
The {\em ring of symmetric functions} 
\begin{equation}
\Lambda \coloneqq \CC[p_1(X), p_2(X), \dots ]
\end{equation}
is the $\CC$-subalgebra of $\CC[[X]]$ freely generated by the $p_d(X)$.
The algebra $\Lambda$ is graded; let $\Lambda_n$ be the subspace of homogeneous degree $n$ 
so that $\Lambda = \bigoplus_{n \geq 0} \Lambda_n$.

For $d \geq 0$ the {\em elementary symmetric function} is
$e_d(X) \coloneqq \sum_{1 \leq i_1 < i_2 < \cdots < i_d}  x_{i_1} x_{i_2} \cdots x_{i_d}$
and the {\em homogeneous symmetric function} is
$h_d(X) \coloneqq \sum_{1 \leq i_1 \leq i_2 \leq \cdots \leq i_d}  x_{i_1} x_{i_2} \cdots x_{i_d}$.
If $\lambda = (\lambda_1, \lambda_2, \dots )$ is a partition, we define
\begin{equation*}
p_{\lambda}(X) \coloneqq p_{\lambda_1}(X) p_{\lambda_2}(X) \cdots, \quad
e_{\lambda}(X) \coloneqq e_{\lambda_1}(X) e_{\lambda_2}(X) \cdots,
\quad \text{ and }\  
h_{\lambda}(X) \coloneqq h_{\lambda_1}(X) h_{\lambda_2}(X) \cdots
\end{equation*}

Given a partition $\lambda$, the {\em Schur function} $s_{\lambda}(X) \in \Lambda$ is the formal power
series $s_{\lambda}(X) \coloneqq \sum_T x^T$, where $T$ ranges over all semistandard tableaux of shape $\lambda$.
The set $\{ s_{\lambda}(X) \,:\, \lambda \vdash n \}$ forms a basis
for $\Lambda_n$ as a $\CC$-vector space.

The irreducible $\symm_n$-modules are naturally indexed by partitions
of $n$.  If $\lambda \vdash n$, let
$S^{\lambda} = \CC[\symm_n] \varepsilon_{\lambda}$ be the
corresponding irreducible module.  If $V$ is any $\symm_n$-module, there
exist unique $m_{\lambda} \geq 0$ such that
$V \cong \bigoplus_{\lambda \vdash n} m_{\lambda} S^{\lambda}$.  The
{\em Frobenius image} $\Frob(V) \in \Lambda_n$ is given by
$\Frob(V) \coloneqq \sum_{\lambda \vdash n} m_{\lambda}
s_{\lambda}(X)$.  A symmetric function $F(X) \in \Lambda_n$ is
therefore Schur positive if and only if $F(X)$ is the Frobenius image
of some $\symm_n$-module $V$.

Let $V$ be an $\symm_n$-module, and let $W$ be an $\symm_m$-module. 
The tensor product $V \otimes W$ is naturally an $\symm_n \times \symm_m$-module.
We have an embedding $\symm_n \times \symm_m \subseteq \symm_{n+m}$ by letting $\symm_n$ act 
on the first $n$ letters and letting $\symm_m$ act on the last $m$ letters.
The {\em induction product} of $V$ and $W$ is
\begin{equation}
V \circ W \coloneqq (V \otimes W) \uparrow_{\symm_n \times \symm_m}^{\symm_{n+m}}.
\end{equation}
The effect of induction product on Frobenius image is
\begin{equation}
\Frob(V \circ W) = \Frob(V) \cdot \Frob(W).
\end{equation}

Suppose $V = \bigoplus_{i \geq 0} V_i$ is a graded $\symm_n$-module such that each
component $V_i$ is finite-dimensional. The {\em graded Frobenius image} is
\begin{equation}
\grFrob(V;t) := \sum_{i \geq 0} \Frob(V_i) \cdot t^i.
\end{equation}
More generally, if $V = \bigoplus_{i, j \geq 0} V_{i,j}$ is a bigraded $\symm_n$-module with 
each component $V_{i,j}$ finite-dimensional, the {\em bigraded Frobenius image} is
\begin{equation}
\grFrob(V;t,q) := \sum_{i,j \geq 0} \Frob(V_{i,j}) \cdot t^i q^j.
\end{equation}

\subsection{The coinvariant algebra}

Let $\CC[X_n]^{\symm_n}_+ \subseteq \CC[X_n]$ be the vector space of 
symmetric polynomials with vanishing constant term and let 
$\langle \CC[X_n]^{\symm_n}_+ \rangle \subseteq \CC[X_n]$ be the ideal generated by this space.
We have the generating set 
$\langle \CC[X_n]^{\symm_n}_+ \rangle = \langle e_1(X_n), e_2(X_n), \dots, e_n(X_n) \rangle$.
The {\em coinvariant ring} is the graded $\symm_n$-module
\begin{equation}
\CC[X_n]/\langle \CC[X_n]^{\symm_n}_+ \rangle  = \CC[X_n]/ \langle e_1(X_n), e_2(X_n), \dots, e_n(X_n) \rangle.
\end{equation}

As an ungraded $\symm_n$-module, the coinvariant ring
is isomorphic to the regular representation $\CC[\symm_n]$.
The graded $\symm_n$-module structure of the coinvariant ring is governed by the combinatorics
of tableaux.

 Let $\SYT(n)$ denote the set of all standard Young
tableaux with $n$ boxes (of any partition shape).  We let $\shape(T)$ be
the partition obtained by erasing the entries of the tableau $T$.  If
$T \in \SYT(n)$, an element $1 \leq i \leq n-1$ is a {\em descent} if
$i+1$ appears in a strictly lower row than $i$ in $T$.  Otherwise, $i$
is an {\em ascent} of $T$.  The {\em major index} $\maj(T)$ is the sum
of the descents in $T$.  For example, the standard tableau on the
right in Figure~\ref{fig:1} has descents at $2,5,$ and $6$ so its
major index is $13$.

The following graded Frobenius image of the coinvariant ring is due to
Lusztig (unpublished) and Stanley \cite{Stanley}.  


\begin{theorem}[Lusztig--Stanley, {\cite[Prop.~4.11]{Stanley}}]\label{thm:LusztigStanley}
  For any positive integer $n$, we have
$$  \grFrob(\CC[X_n]/\langle \CC[X_n]^{\symm_n}_+ \rangle ; t) 
= \sum_{T \in \SYT(n)} t^{\maj(T)} \cdot s_{\shape(T)}(X).$$
\end{theorem}

\section{Chern plethysm and Schur positivity}
\label{Chern}

\subsection{Chern classes and Chern roots}
We describe basic properties of vector bundles and their Chern roots
from a combinatorial point of view; see
\cite{FultonIntersection}  for the relevant geometry.

Let $X$ be a smooth complex projective variety, and let $H^{\bullet}(X)$ be the singular 
cohomology of $X$ with integer coefficients. Let
$\EEE \twoheadrightarrow X$ be a complex vector 
bundle over $X$ of rank $n$. For any point $p \in X$, the fiber $\EEE_p$ of $\EEE$ over $p$
is an $n$-dimensional complex vector space.  For $1 \leq i \leq r$ we have the {\em Chern class}
$c_i(\EEE) \in H^{2i}(X)$.  
The sum of these Chern classes inside $H^{\bullet}(X)$ is the {\em total Chern class}
$c_{\bullet}(\EEE) \coloneqq 1 + c_1(\EEE)  + c_2(\EEE)  + \cdots + c_n(\EEE) $.

If $\EEE \twoheadrightarrow X$ and $\FFF \twoheadrightarrow X$ are two vector bundles, we can form their 
{\em direct sum bundle} (or {\em Whitney sum}) by the rule 
$(\EEE \oplus \FFF)_p \coloneqq \EEE_p \oplus \FFF_p$ for all $p \in X$.
The ranks of these bundles are related by  $\rank(\EEE \oplus \FFF) = \rank(\EEE) + \rank(\FFF)$.
The {\em Whitney sum formula} 
\cite[Thm. 3.2e]{FultonIntersection} 
states that the corresponding total Chern classes are related by 
$c_{\bullet}(\EEE \oplus \FFF) = c_{\bullet}(\EEE) \cdot c_{\bullet}(\FFF)$.

Recall that a {\em line bundle} is a vector bundle of rank $1$.
Let $\EEE \twoheadrightarrow X$ be a rank $n$ vector bundle.
If we can express $\EEE$ as a direct sum of $n$ line bundles, i.e. $\EEE = \ell_1 \oplus \cdots \oplus \ell_n$,
then the Whitney sum formula guarantees that the total Chern class of $\EEE$ factors as
\begin{equation}
\label{chern-split}
c_{\bullet}(\EEE) = (1 + x_1) \cdots (1 + x_n) 
\end{equation}
where 
$x_i = c_1(\ell_i) \in H^2(X)$.
Despite the fact $\EEE$ is not necessarily a direct sum of line bundles, by the {\em splitting principle}
\cite[Rmk. 3.2.3]{FultonIntersection}
there exist, unique up to permutation, elements 
$x_1, \dots, x_n \in H^2(X)$ such that the factorization~\eqref{chern-split} holds.
The elements $x_1, \dots, x_n$ of the second cohomology group of $X$ are the
{\em Chern roots} of the bundle $\EEE$.

\subsection{Chern plethysm}  Let $\EEE \twoheadrightarrow X$ be a rank $n$ complex vector bundle 
over a smooth algebraic
variety, and let $x_1, \dots, x_n$ be the Chern roots of $\EEE$.  If $F \in \Lambda$ is a symmetric
function, we define the {\em Chern plethysm} $F(\EEE)$ to be the result of plugging the Chern roots
$x_1, \dots, x_n$ of 
$\EEE$ into $n$ of the arguments of $F$, and setting all other arguments of $F$ equal to zero.
Informally, the expression $F(\EEE)$ evaluates $F$ at the Chern roots of $\EEE$.
If $F \in \Lambda_d$ is homogeneous of degree $d$, then $F(\EEE) = F(x_1, \dots, x_n) \in H^{2d}(X)$ is a polynomial
of degree $d$ in the $x_i$; the degree $d$ is independent of the rank of $\EEE$.

Performing fiberwise operations on vector bundles induces linear changes in their Chern roots.
We give three examples of how this applies to Chern plethysm.
Let $\EEE$ be a vector bundle with Chern roots $x_1, \dots, x_n$ and let $\FFF$ be a vector bundle with
Chern roots $y_1, \dots, y_m$.  

\begin{itemize}
\item
By the Whitney sum formula, the Chern roots of the direct sum
$\EEE \oplus \FFF$ are the multiset union of $\{x_1, \dots, x_n\}$ and $\{y_1, \dots, y_m\}$ so that
\begin{equation}
F(\EEE \oplus \FFF) = F(x_1, \dots, x_n, y_1, \dots, y_m).
\end{equation}
\item
The tensor product bundle $\EEE \otimes \FFF$ is defined by
$(\EEE \otimes \FFF)_p \coloneqq \EEE_p \otimes \FFF_p$ for all $p \in X$.
The Chern roots of $\EEE \otimes \FFF$ are the multiset of sums $x_i + y_j$ where $1 \leq i \leq n$ and
$1 \leq j \leq m$ so that
\begin{equation}
F(\EEE \otimes \FFF) = F( \overbrace{\dots, \, x_i + y_j, \, \dots }^{1 \leq i \leq n, \, \, 1 \leq j \leq m}).
\end{equation}
Since $F$ is symmetric, the ordering of these arguments is immaterial.
\item
Let $\lambda$ be a partition. We may apply the Schur functor $\SSS^{\lambda}$
to the bundle $\EEE$ to obtain a new bundle $\SSS^{\lambda}(\EEE)$ with fibers
$\SSS^{\lambda}(\EEE)_p \coloneqq \SSS^{\lambda}(\EEE_p)$.  The Chern roots of $\SSS^{\lambda}(\EEE)$
are the multiset of sums 
$\sum_{\square \in \lambda} x_{T(\square)}$ where $T$ varies over $\SSYT(\lambda, \leq n)$ so that
\begin{equation}
F(\SSS^{\lambda}(\EEE)) = 
F( \overbrace{\dots, \, \sum_{\square \in \lambda} x_{T(\square)},  \, \dots}^{T \in \SSYT(\lambda, \leq n)} )
\end{equation}
For example, if $\lambda = (2,1)$ and $n = 3$, the elements of $\SSYT(\lambda, \leq n)$ are 
\begin{equation*}
\begin{young}
1 & 1 \cr
2
\end{young} \hspace{0.2in}
\begin{young}
1 & 2 \cr
2
\end{young} \hspace{0.2in}
\begin{young}
1 & 1 \cr
3
\end{young} \hspace{0.2in}
\begin{young}
1 & 2 \cr
3
\end{young} \hspace{0.2in}
\begin{young}
1 & 3 \cr
2
\end{young} \hspace{0.2in}
\begin{young}
2 & 2 \cr
3
\end{young} \hspace{0.2in}
\begin{young}
1 & 3 \cr
3
\end{young} \hspace{0.2in}
\begin{young}
2 & 2 \cr
3
\end{young}
\end{equation*}
and $F(\SSS^{\lambda}(\EEE))$ is the expression
\begin{equation*}
F(2x_1 + x_2, x_1 + 2x_2, 2x_1 + x_3, x_1 + x_2 + x_3, x_1 + x_2 + x_3, 2 x_2 + x_3, x_1 + 2 x_3, 2 x_2 + x_3).
\end{equation*}
As before, the symmetry of $F$ makes the order of substitution irrelevant.
\end{itemize}

If $F, G \in \Lambda$ are any symmetric functions and $\alpha, \beta \in \CC$ are scalars, we have 
the laws of polynomial evaluation
\begin{equation}
\label{chern-plethysm-relations}
\begin{cases}
(F \cdot G)(\EEE) = F(\EEE) \cdot G(\EEE), \\
(\alpha F + \beta G)(\EEE) = \alpha F(\EEE) + \beta G(\EEE), \\
\alpha(\EEE) = \alpha
\end{cases}
\end{equation}
for any vector bundle $\EEE$.

\begin{remark}
The   reader may worry that, since $F(\EEE)$ lies in the cohomology ring $H^{\bullet}(X)$ of the base 
space $X$ of the bundle $\EEE$, relations in $H^{\bullet}(X)$ may preclude the use of Chern plethysm of 
proving that polynomials in the ring $\CC[x_1, \dots, x_n]^{\symm_n}$ are Schur positive. Fortunately,
the base space $X$ may be chosen so that the Chern roots $x_1, \dots, x_n$ are algebraically independent.
For example, one can take $X$ to be the $n$-fold product $\PP^{\infty} \times \cdots \times \PP^{\infty}$
of infinite-dimensional complex projective space with itself and let $\EEE = \ell_1 \oplus \cdots \oplus \ell_n$
be the direct sum of the tautological line bundles over the $n$ factors of $X$. 
 The Chern roots of $\EEE$ are the variables
$x_1, \dots, x_n$ in the  presentation $H^{\bullet}(X) = \ZZ[x_1, \dots, x_n]$.
For this reason, there is no harm done in thinking of $F(\EEE)$ as an honest symmetric polynomial
in $\CC[x_1, \dots, x_n]^{\symm_n}$.
\end{remark}

\subsection{Comparison with classical plethysm} Given a symmetric function $F$ and any rational
function $E = E(t_1, t_2, \dots )$ in a countable set of variables, there is a classical notion of plethysm
$F[E]$.  The quantity $F[E]$ is determined by imposing the same relations as \eqref{chern-plethysm-relations},
i.e.
\begin{equation}
\label{plethysm-relations}
\begin{cases}
(F \cdot G)[E] = F[E] \cdot G[E], \\
(\alpha F + \beta G)[E] = \alpha F[E] + \beta G[E], \\
\alpha[E] = \alpha
\end{cases}
\end{equation}
for all $F, G \in \Lambda$ and $\alpha, \beta \in \CC$ 
together with the condition
\begin{equation}
p_k[E] = p_k[E(t_1, t_2, \dots )] \coloneqq E(t_1^k, t_2^k, \dots ), \quad k \geq 1.
\end{equation}
Since the power sums $p_1, p_2, \dots $ freely generate $\Lambda$ as a $\CC$-algebra, this defines $F[E]$ uniquely.
For more information on classical plethysm, see \cite{qtbook}.

Let us compare the two notions of plethysm $F(\EEE)$ and $F[E]$.
For any bundle $\EEE$, the degree of the polynomial $F(\EEE)$ equals the degree $\deg(F)$ of $F$.
However, if $E$ is a polynomial (or formal power series) of degree $e$, the degree of $F[E]$
is $e \cdot \deg(F)$. 

If $x_1, \dots, x_n$ are the Chern roots of $\EEE$ we have the relation
\begin{equation}
F(\EEE) = F[x_1 + \cdots + x_n] = F[X_n]
\end{equation}
for any symmetric function $F$,
where we adopt the plethystic shorthand $X_n = x_1 + \dots + x_n$ for a sum over an alphabet of $n$ variables.
The direct sum operation on vector bundles corresponds to classical plethystic sum in the sense that if $\FFF$ 
is another vector bundle with Chern roots $y_1, \dots, y_m$ then 
\begin{equation}
F(\EEE \oplus \FFF) = F[x_1 + \cdots + x_n + y_1 + \cdots + y_m] = F[X_n + Y_m].
\end{equation}
However, there is no natural interpretation of $F(\EEE \otimes \FFF)$ 
or $F(\SSS^{\lambda}(\EEE))$ in terms of classical plethysm.
On the other hand, there does not seem to be a natural interpretation of expressions like 
$F[X_n \cdot Y_m] = F(\dots, x_i y_j, \dots )$ in terms of Chern plethysm.

Classical plethystic calculus is among the most powerful tools in symmetric function theory (see e.g. \cite{CM}).
In this paper we will use geometric results to prove the Schur positivity of polynomials coming 
from Chern plethysm. It is our hope that Chern plethystic calculus will prove useful in the future.

\subsection{Chern plethysm and Schur positivity}
We have the following positivity result of Pragacz, stated in the language of Chern plethysm.

\begin{theorem}
\label{pragacz-theorem}
(Pragacz \cite[Cor. 7.2]{PSymmetric}, see also \cite[p. 34]{PAlain})
Let $\EEE_1, \dots, \EEE_k$ be vector bundles and let $\lambda, \mu^{(1)}, \dots, \mu^{(k)}$ be partitions.
There exist nonnegative integers $c^{\lambda, \mu^{(1)}, \dots, \mu^{(k)}}_{\nu^{(1)}, \dots, \nu^{(k)}}$ so that 
\begin{equation*}
s_{\lambda}(\SSS^{\mu^{(1)}}(\EEE_1) \otimes \cdots \otimes \SSS^{\mu^{(k)}}(\EEE_k)) =
\sum_{\nu^{(1)}, \dots, \nu^{(k)}}
c^{\lambda, \mu^{(1)}, \dots, \mu^{(k)}}_{\nu^{(1)}, \dots, \nu^{(k)}} \cdot
 s_{\nu^{(1)}}(\EEE_1) \cdots s_{\nu^{(k)}}(\EEE_k).
\end{equation*}
\end{theorem} 

Pragacz's Theorem~\ref{pragacz-theorem} relies on deep work of Fulton and Lazarsfeld \cite{FL}
in the context of numerical positivity. The Hard Lefschetz Theorem is a key tool in \cite{FL}.

We are ready to deduce the Schur positivity of the Boolean product polynomials.

\begin{theorem}
\label{boolean-schur-positivity}
The Boolean product polynomials $B_{n,k}(X_n)$ 
and $B_n(X_n)$ are Schur positive.
\end{theorem}

\begin{proof}
  Let $\EEE \twoheadrightarrow X$ be a rank $n$ vector bundle over a
  smooth variety $X$ with Chern roots $x_1, \dots, x_n$.  The $k^{th}$
  exterior power $\wedge^k \EEE$ has Chern roots
  $\{ x_{i_1} + \cdots + x_{i_k} \,:\, 1 \leq i_1 < \cdots < i_k \leq
  n \}$.  By Pragacz's Theorem~\ref{pragacz-theorem}, the polynomial
  $s_{\lambda}(\wedge^k \EEE )$ is Schur positive for any partition
  $\lambda$.  In particular, if we take $\lambda = (1, \dots, 1)$ to
  be a single column of size ${n \choose k}$, we have
 \begin{equation}
 s_{\lambda}\left( \wedge^k \EEE \right) = \prod_{1 \leq i_1 < \cdots < i_k \leq n} (x_{i_1} + \cdots + x_{i_k}) =
 B_{n,k}(X_n),
\end{equation}
so $ B_{n,k}(X_n)$ has a Schur positive expansion.  Since
$B_n(X_n) = \prod_{1 \leq k \leq n} B_{n,k}(X_n)$, the
Littlewood-Richardson rule implies that $B_n(X_n)$ is also Schur
positive.
\end{proof}

As stated in the introduction, it is easy to determine the Schur
expansion explicitly for $B_{n,1}(X_n)=s_{(1^n)}(X_n)$ and
$B_{n,2}(X_n)=s_{(n-1,n-2,\ldots,1)}$ for $n\geq 2$.  We will discuss
the Schur expansion for $B_{n,n-1}(X_n)$ in Section~\ref{Analogue}.
No effective formula for the Schur expansion of $B_{n,k}(X_n)$ is
known for $3\leq k\leq n-2$.  

\begin{problem}
\label{combinatorial-boolean-problem}
Find a combinatorial interpretation for the coefficients in the Schur
expansions of $B_{n,k}(X_n)$ and $B_n(X_n)$.  
\end{problem}

Theorem~\ref{boolean-schur-positivity} guarantees the existence of $GL_n$-modules whose
Weyl characters are $B_{n,k}(X_n)$ and $B_n(X_n)$.

\begin{problem}
\label{gl-module-problem}
Find natural $GL_n$-modules $V_{n,k}$ and $V_n$ such that 
$\ch(V_{n,k}) = B_{n,k}(X_n)$ and $\ch(V_n) = B_n(X_n)$.
\end{problem}

If $U$ and $W$ are $GL_n$-modules and we endow $U \otimes W$ with the diagonal action
$g.(u \otimes w) \coloneqq (g.u) \otimes (g.w)$ of $GL_n$, then
$\ch(U \otimes W) = \ch(U) \cdot \ch(W)$.
If we can find a module $V_{n,k}$ as in Problem~\ref{gl-module-problem},
 we can therefore take $V_n = V_{n,1} \otimes V_{n,2} \otimes \cdots \otimes V_{n,n}$.
 
 If $V_{n,k}$ is a $GL_n$-module as in Problem~\ref{gl-module-problem}, then we must have
 \begin{equation}
 \dim(V_{n,k}) = B_{n,k}(x_1, \dots, x_n) \mid_{x_1 = \cdots = x_n = 1} = k^{{n \choose k}}.
 \end{equation}
 A natural vector space of this dimension may be obtained as follows. Let $\ee_1, \dots, \ee_n$ be the 
 standard basis of $\CC^n$. For any subset $I \subseteq [n]$, let $\CC^I$ be the span 
 of $\{ \ee_i \,:\, i \in I \}$.  Then
 \begin{equation}
 V_{n,k} \coloneqq \bigotimes_{I \subseteq [n], \, \, |I| = k} \CC^I
 \end{equation}
 is a vector space of the correct dimension.
The action of the diagonal subgroup $(\CC^{\times})^n \subseteq GL_n$ 
preserves each tensor factor of $V_{n,k}$, and we have
\begin{equation}
\mathrm{trace}_{V_{n,k}} ( \mathrm{diag}(x_1, \dots, x_n)) = B_{n,k}(x_1, \dots, x_n).
\end{equation}
One way to solve Problem~\ref{gl-module-problem} would be to extend this action to the full
general linear group $GL_n$.

\subsection{Bivariate Boolean Product Polynomials}
What happens when we apply Pragacz's Theorem~\ref{pragacz-theorem} to the case
of more than one vector bundle $\EEE_i$?
This yields Schur positivity results involving polynomials over more than one set of variables.
For clarity, we describe the case of two bundles here.

Let $\EEE$ be a vector bundle with Chern roots $x_1, \dots, x_n$ and $\FFF$ be a vector bundle 
with Chern roots $y_1, \dots, y_m$.  For $1 \leq k \leq n$ and $1 \leq \ell \leq m$, we have the extension
 of the Boolean product polynomial to two sets of variables
\begin{equation}
P_{k,\ell}(X_n; Y_m) \coloneqq \prod_{1 \leq i_1 < \cdots < i_k \leq n}\
\prod_{1 \leq j_1 < \cdots < j_{\ell} \leq m} (x_{i_1} + \cdots + x_{i_k}+ 
y_{j_1} + \cdots + y_{j_{\ell}}).
\end{equation}
Observe that $P_{k,\ell}$ equals the Chern plethysm $e_d( \wedge^k \EEE \otimes \wedge^{\ell} \FFF)$,
where $d = {n \choose k} {m \choose \ell}$.  By Theorem~\ref{pragacz-theorem}, there are nonnegative
integers $a_{\lambda,\mu}$ such that
\begin{equation}
\label{bivariate-identity}
P_{j,\ell}(X_n; Y_m) = \sum_{\lambda, \mu} a_{\lambda,\mu} \cdot
s_{\lambda}(X_n) \cdot s_{\mu}(Y_m).
\end{equation}
Setting the $y$-variables equal to zero recovers Theorem~\ref{boolean-schur-positivity}.

Equation~\eqref{bivariate-identity} is reminiscent of the {\em dual Cauchy identity} which 
uses the Robinson-Schensted-Knuth correspondence to give a combinatorial proof that
\begin{equation}
\label{dual-cauchy}
\prod_{1\leq i\leq n}\prod_{1\leq j \leq m} (x_i + y_j) = \sum_{\lambda \subseteq (m^n)} s_{\lambda}(X_n) \cdot s_{\tilde{\lambda}}(Y_m),
\end{equation}
where $\tilde{\lambda}$ is the transpose of the complement of $\lambda$ inside the rectangular Ferrers shape $(m^n)$.
This raises the following natural problem.

\begin{problem}
\label{rsk-problem}
Develop a variant of the RSK correspondence which proves the integrality and nonnegativity of the 
$a_{\lambda,\mu}$ in Equation~\eqref{bivariate-identity}.
\end{problem}



\section{A combinatorial interpretation of Lascoux's Formula}\label{sec:bn2}

Pragacz's Theorem has the following sharpening due to Lascoux in the
case of one vector bundle.  In fact, Lascoux's Theorem was part of the
inspiration for Pragacz's Theorem.  

\begin{theorem}
\label{lascoux-theorem} (Lascoux \cite{Lascoux}) Let $\EEE$ be a rank $n$ vector bundle
with Chern roots $x_1, \dots, x_n$, so that we have the total Chern classes
\begin{align*}
c(\wedge^2 \EEE) =&
\prod_{1 \leq i < j \leq n} (1 + x_i + x_j), \text{ and} \\
c(Sym^2 \EEE) =&
\prod_{1 \leq i \leq j \leq n} (1 + x_i + x_j).
\end{align*}
Let
$\delta_n \coloneqq (n, n-1, \dots, 1)$ be the staircase partition with largest part $n$.
There exist integers $d^{(n)}_{\lambda, \mu}$ for $\mu \subseteq \lambda$ such that 
\begin{align*}
&  \prod_{1 \leq i < j \leq n} (1 + x_i + x_j) = 2^{-{n \choose 2}}
  \sum_{\mu \subset \delta_{n-1} } 2^{|\mu|} \cdot
  d^{(n)}_{\delta_{n-1},\mu} \cdot s_{\mu}(x_1, \dots, x_n), \text{ and}\\
&  \prod_{1 \leq i \leq j \leq n} (1 + x_i + x_j) = 2^{-{n \choose 2}}
  \sum_{\mu \subset \delta_n} 2^{|\mu|} \cdot d^{(n)}_{\delta_n, \mu} \cdot
  s_{\mu}(x_1, \dots, x_n),
\end{align*}
\end{theorem}

The integers $d^{(n)}_{\lambda,\mu}$ of Theorem~\ref{lascoux-theorem} are given as follows.  Pad $\lambda$
and $\mu$ with 0's so that both sequences $(\lambda_1, \dots, \lambda_n)$
and $(\mu_1, \dots, \mu_n)$ have length $n$.  Assuming $\mu \subseteq \lambda$, the integer $d^{(n)}_{\lambda,\mu}$
is the following determinant of binomial coefficients
\begin{equation}\label{eq:las.det}
d^{(n)}_{\lambda, \mu} = \det \left( {\lambda_i + n - i \choose \mu_j + n - j} \right)_{1 \leq i,j \leq n}.
\end{equation}
The positivity of this determinant is not obvious. Lascoux \cite{Lascoux} gave a geometric proof that 
$d^{(n)}_{\lambda,\mu} \geq 0$.
This determinant 
was a motivating example for the seminal work of Gessel and Viennot \cite{GV}; they gave an interpretation
of $d^{(n)}_{\lambda,\mu}$ (and many other such determinants) 
as counting  families of nonintersecting lattice paths.

By the work of Lascoux \cite{Lascoux} or Gessel-Viennot \cite{GV}, the Schur expansions of
Theorem~\ref{lascoux-theorem} have nonnegative rational coefficients.
In order to deduce that these coefficients are in fact nonnegative integers, observe that the monomial expansions
of $\prod_{1 \leq i < j \leq n}(1 + x_i + x_j)$ and
$\prod_{1 \leq i \leq j \leq n}(1 + x_i + x_j)$ visibly have positive integer coefficients.  Since the transition matrix
from the monomial symmetric functions to the Schur functions (the {\em inverse Kostka matrix}) has integer entries
(some of them negative),
we conclude that the Schur expansions of Theorem~\ref{lascoux-theorem}
have integer coefficients.  For example,
\begin{equation}\label{eq:n=3}
  \prod_{1 \leq i < j \leq 3} (1 + x_i + x_j) =
1+2s_{(1)}(X_3) + s_{(2)}(X_3) + 2s_{(1,1)}(X_3) +  s_{(2,1)}(X_3).
\end{equation}

A manifestly integral and positive formula for the Schur expansions in 
Theorem~\ref{lascoux-theorem} may be given as follows.
For a partition $\mu \subseteq \delta_{n-1}$, a filling $T: \mu \rightarrow \ZZ_{\geq 0}$ 
is {\em reverse flagged} if 
\begin{itemize}
\item the entries of $\mu$ decrease strictly across rows and weakly down columns, and
\item the entries in row $i$ of $\mu$ lie between $1$ and $n-i$.
\end{itemize}
Let $r^{(n)}_{\mu}$ be the number of reverse flagged fillings of shape $\mu$.
In the case $n = 3$, the collection of reverse flagged fillings of shapes $\mu \subseteq \delta_2 = (2,1)$ are as follows:
\begin{equation*}
\varnothing \quad
\begin{young}
1
\end{young} \quad
\begin{young}
2
\end{young} \quad
\begin{young}
2 & 1
\end{young} \quad
\begin{young}
2 \cr 1
\end{young} \quad
\begin{young}
1 \cr 1
\end{young} \quad
\begin{young}
2 & 1 \cr 1
\end{young} \quad
\end{equation*}
Compare the shapes in this example to the expansion in \eqref{eq:n=3}.   

\begin{theorem}
\label{lascoux-schur-expansion}
For $n \geq 1$ we have the Schur expansions
\begin{align}\label{eq:lascoux.1}
\prod_{1 \leq i < j \leq n} (1 + x_i + x_j) &= \sum_{\mu \subseteq \delta_{n-1}} r^{(n)}_{\mu} \cdot s_{\mu}(X_n), \\
\prod_{1 \leq i \leq j \leq n} (1 + x_i + x_j) &= \sum_{\lambda \subseteq \delta_{n}}
\sum_{\substack{\mu \subseteq \lambda \cap \delta_{n-1} \\ \lambda/\mu \text{ a vertical strip}}} 
2^{|\lambda/\mu|} 
r^{(n)}_{\mu} \cdot s_{\lambda}(X_n). \label{eq:lascoux.2}
\end{align}
Recall that the set-theoretic difference
 $\lambda/\mu$ of Ferrers diagrams
  is a {\em vertical strip} if no row of $\lambda/\mu$ contains more than one box.
\end{theorem}

\begin{proof}
We have 
\begin{equation}
\prod_{1 \leq i \leq j \leq n} (1 + x_i + x_j) = \left[ \sum_{r = 0}^n 2^r \cdot e_r(X_n) \right] \cdot 
\prod_{1 \leq i < j \leq n} (1 + x_i + x_j).
\end{equation}
By the {\em dual Pieri rule}, the Schur expansion of $e_r(X_n) \cdot s_{\mu}(X_n)$ is obtained by adding
a vertical strip of size $r$ to $\mu$ in all possible ways, so the second equality follows from the first.

We start with the observation
\begin{equation}
\prod_{1 \leq i < j \leq n} (1 + x_i + x_j)  = \prod_{1 \leq i < j \leq n}
\frac{x_i(1+x_i) - x_j(1+x_j)}{x_i - x_j}.
\end{equation}
Comparing this product with the formula for the Vandermonde
determinant and using the antisymmetrizing operators $A_n$ defined in
\eqref{eq:anti.symm} , we have
\begin{equation}
\label{antisymmetrizer-interpretation} \prod_{1 \leq i < j \leq n} (1
+ x_i + x_j) = A_n \left( \prod_{i = 1}^n x_i^{n-i} (1 + x_i)^{n-i}
\right).
\end{equation}
The antisymmetrizing operator acts linearly, so
$\prod_{1 \leq i < j \leq n} (1 + x_i + x_j)$ is the positive sum of
terms of the form
$c_\alpha A_n(x_1^{\alpha_1}x_2^{\alpha_2} \cdots x_n^{\alpha_n})$.  If
$\alpha_i=\alpha_j$ for $i\neq j$, then
$A_n(x_1^{\alpha_1}x_2^{\alpha_2} \cdots x_n^{\alpha_n})=0$.  If the
$\alpha_i$'s are all distinct, then there exists a permutation
$w \in \symm_n$ and a partition $\mu= (\mu_1, \dots, \mu_n)$ such that

$$x_1^{\alpha_1}x_2^{\alpha_2} \cdots x_n^{\alpha_n} = x_1^{\mu_{w(1)}
  + n - w(1)} \cdots x_n^{\mu_{w(n)} + n - w(n)},$$ hence
$A_n(x_1^{\alpha_1}x_2^{\alpha_2} \cdots
x_n^{\alpha_n})=\sign(w)s_\mu(X_n).$ Therefore, in
Equation~\eqref{antisymmetrizer-interpretation}, the coefficient of
$s_{\mu}(X_n)$ in the Schur expansion of
$\prod_{1 \leq i < j \leq n} (1 + x_i + x_j)$ is given by the signed
sum
\begin{equation}
\label{signed-sum}
\sum_{w \in \symm_n} \sign(w)  \cdot
\left(
\begin{array}{c}
\text{coefficient of 
$x_1^{\mu_{w(1)} + n - w(1)} \cdots x_n^{\mu_{w(n)} + n - w(n)}$} \\
\text{in $\prod_{i = 1}^n x_i^{n-i} (1 + x_i)^{n-i}$}
\end{array}
\right).
\end{equation}
By the Binomial Theorem, the expression in \eqref{signed-sum} is equal to
\begin{equation}
\label{binomial-signed-sum}
\sum_{w \in \symm_n} \sign(w)  \cdot
{n-1 \choose \mu_{w(1)} - w(1) + 1)} 
{n-2 \choose \mu_{w(2)} - w(2) + 2}  \cdots 
{n-n \choose \mu_{w(n)} - w(n) + n}.
\end{equation}
In turn, the expression in Equation~\eqref{binomial-signed-sum} equals the determinant of binomial coefficients
\begin{equation}
\label{binomial-determinant}
\det \left(
{n-i \choose \mu_j - j + i}
\right)_{1 \leq i, j \leq n}.
\end{equation}

We must show that the determinant \eqref{binomial-determinant} counts
reverse flagged fillings of shape $\mu = (\mu_1, \dots, \mu_n)$ with
rows bounded by $(n-1, n-2, \ldots, 0)$.  To do this, we use
Gessel-Viennot theory \cite{GV}, see also \cite[Sec. 2.7]{ec1}.  For
$1 \leq i \leq n$, define lattice points $p_i$ and $q_i$ by
$p_i = (2i - 2, n - i)$ and
$q_i = (n + i - \mu_i - 2, n - i + \mu_i)$.  The number of paths from
$p_i$ to $q_j$ is the binomial coefficient
${n - i \choose \mu_j - j + i}$, which is the $(i,j)$-entry of the
determinant \eqref{binomial-determinant}.  It follows that the
determinant \eqref{binomial-determinant} counts nonintersecting path
families $\mathbb{L}=(L_1, \ldots, L_n)$ such that $L_i$ connects
$p_i$ to $q_i$ for all $1 \leq i \leq n$; one such nonintersecting
path family is shown below in the $n = 5$ and $\mu = (2,2,1,1,0)$. 
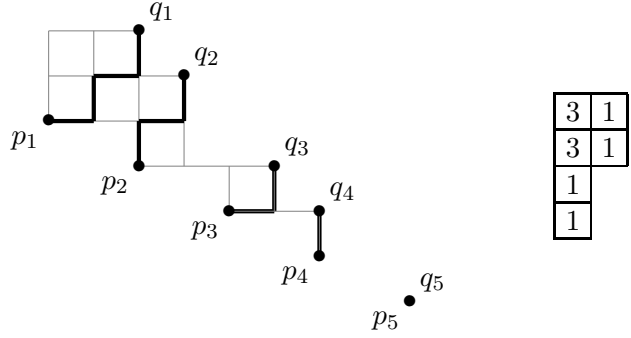
\begin{center}
\begin{tikzpicture}[scale = 0.6]

\draw [help lines] (0,4) -- (0,6);
\draw [help lines] (0,4) -- (3,4);
\draw [help lines] (0,5) -- (3,5);
\draw [help lines] (0,6) -- (2,6);
\draw [help lines] (1,4) -- (1,6);
\draw [help lines] (2,3) -- (2,6);
\draw [help lines] (2,3) -- (5,3);
\draw [help lines] (3,3) -- (3,5);

\draw [ultra thick] (0,4) -- (1,4);
\draw [ultra thick] (1,4) -- (1,5);
\draw [ultra thick] (1,5) -- (2,5);
\draw [ultra thick] (2,5) -- (2,6);

\draw [ultra thick] (2,3) -- (2,4);
\draw [ultra thick] (2,4) -- (3,4);
\draw [ultra thick] (3,4) -- (3,5);

\draw [ultra thick] (4,2) -- (5,2);
\draw [ultra thick] (5,2) -- (5,3);

\draw [ultra thick] (6,1) -- (6,2);

\draw [help lines] (4,2) -- (4,3);
\draw [help lines] (4,2) -- (6,2);
\draw [help lines] (5,2) -- (5,3);

\draw [help lines] (6,1) -- (6,2);

\node at (8,0) {$\bullet$};
\node [above right] at (8,0) {$q_5$};
\node [below left] at (8,0) {$p_5$};
\node at (6,2) {$\bullet$};
\node [above right] at (6,2) {$q_4$};
\node [below left] at (6,1) {$p_4$};
\node at (0,4) {$\bullet$};
\node [below left] at (0,4) {$p_1$};
\node at (5,3) {$\bullet$};
\node [above right] at (5,3) {$q_3$};
\node [above right] at (3,5) {$q_2$};
\node [above right] at (2,6) {$q_1$};
\node at (2,3) {$\bullet$};
\node [below left] at (4,2) {$p_3$};
\node [below left] at (2,3) {$p_2$};
\node at (3,5) {$\bullet$};
\node at (4,2) {$\bullet$};
\node at (2,6) {$\bullet$};
\node at (6,1) {$\bullet$};

\node at (12,3)
{\begin{young} 3 & 1 \cr 3 & 1 \cr 1 \cr  1\cr  \end{young}};
\end{tikzpicture}
\end{center}

For the final step proving the theorem, we will show there is a
bijection from the family of nonintersecting lattice paths with
starting points $(p_1, \ldots, p_n)$ and ending points
$(q_1,\ldots, q_n)$ to reverse flagged fillings of shape $\mu$ with
rows bounded by $(n-1,n-2,\ldots, 0)$.  Let
$\mathbb{L}=(L_1,L_2,\ldots, L_n)$ be such a path family.  For
$1 \leq i \leq n$, label the edges of the lattice path $L_i$ in order
by $1,2,\ldots, n-i$ starting at $q_i$ and proceeding southwest toward
$p_i$.  Observe, the lattice path $L_i$ is completely determined by
the subset $R_i$ of edge labels on its vertical edges.  Furthermore,
$|R_i|=\mu_i$ for each $i$.  Let $F(\mathbb{L})$ be the filling of
$\mu$ with $i$th row given by $R_i$ written in decreasing order from
left to right.  By construction, the entries in the $i$th row are
between $1$ and $n-i$.  One can check the nonintersecting condition is
equivalent to the condition that the columns of $F(\mathbb{L})$ are
weakly decreasing.  The inverse map is similarly easy to construct
from the rows of a reverse flagged filling.  Thus, $F$ is the desired
bijection, and the expansion in \eqref{eq:lascoux.1} holds.  An
example of this correspondence is shown above.
\end{proof}

We note that the number of reverse flagged fillings of $\mu$ is
effectively calculated by the binomial determinant given in
\eqref{binomial-determinant} and by Lascoux's formula
\eqref{eq:las.det}.

\begin{corollary}\label{cor:counting_good_fillings}
For $\mu=(\mu_1,\ldots,\mu_n) \subset \delta_{n-1}$, 
$$ r^{(n)}_{\mu} = \det \left(
  {n-i \choose \mu_j - j + i} \right)_{1 \leq i, j \leq n} =
\frac{2^{|\mu|} }{2^{-{n \choose 2}}} \cdot \det \left( {2n - 2i \choose
    \mu_j + n - j} \right)_{1 \leq i,j \leq n} . $$
\end{corollary}

We also have the following curious relationship between the
coefficients in the Schur expansion of
$\prod_{1 \leq i < j \leq n} (1 + x_i + x_j)$ and alternating sign
matrices, \cite[A005130]{oeis}.   

\begin{corollary}\label{cor:total_good_fillings}
  The sum
  $\sum_{\mu\subseteq \delta_{n-1}} r^{(n)}_{\mu} = \prod_{k=0}^{n-1}
  (3k+1)!/(n+k)!$ which is the number of $n\times n$ alternating sign
  matrices.
\end{corollary}
\begin{proof}
  From the first determinantal expression for $r^{(n)}_{\mu}$ in
  Corollary~\ref{cor:counting_good_fillings}, we have
\begin{align}\label{eqn:di_francesco}
\sum_{\mu\subseteq \delta_{n-1}}r^{(n)}_{\mu}=\sum_{\mu\subseteq \delta_{n-1}}\det \left(
{n-i \choose \mu_j - j + i}
  \right)_{1 \leq i, j \leq n}.
\end{align}
Note, the bottom row of each binomial determinant is all zeros except
the $(n,n)$ entry which is 1, hence without changing the sum we can
restrict to the determinants of $(n-1)\times (n-1)$ matrices.  If we
define $r_i=\mu_i+n-i$ for $1\leq i\leq n-1$, then the expression on
the right hand side (up to a minor reindexing) in
Equation~\eqref{eqn:di_francesco} is the expression in \cite[Equation
3.1]{DiFr},
which enumerates the number of totally symmetric
self-complementary plane partitions of $2n$. Such partitions are known
to be equinumerous with the set of $n\times n$ alternating sign
matrices \cite{Andrews}, and the claim follows.
\end{proof}

\begin{remark}
Reverse flagged fillings show up in Kirillov's work in disguise \cite{Kirillov}, albeit with different motivation. 
More precisely, Kirillov \cite[Section 5.1]{Kirillov} considers fillings of shape $\mu\subseteq \delta_{n-1}$ that increase weakly along rows and strictly along columns and further satisfy the property that  entries in row $i$ belong to the interval $[i,n-1]$. 
These fillings can be transformed to our reverse flagged fillings by taking transposes and changing each entry $j$ to $n-j$. In view of this transformation,
our determinantal formula for reverse flagged fillings is the same as the determinant present in \cite[Theorem 5.6]{Kirillov}. 
\end{remark}

\begin{remark}
  The sum of the coefficients in \eqref{eq:lascoux.2} also give rise
  to an integer sequence $f(n)$ starting 3, 16, 147, 2304, 61347.
  This appears to be a new sequence in the literature
  \cite[A306397]{oeis}.  We can give this sequence the following
  combinatorial interpretation.  If we denote the number of 1s in a
  reverse flagged filling $T$ by $m_1(T)$, then this sum of
  coefficients can be written as
\begin{align}
f(n)=\sum_{\lambda\subseteq \delta_n}\sum_{T}2^{m_1(T)}.
\end{align}
Here the inner sum runs over all reverse flagged fillings $T$ of shape
$\lambda$.  
\end{remark}

\section{A $q$-analogue of $B_{n,n-1}$ and superspace}
\label{Analogue}

In this section, we give representation theoretic models for $B_{n,k}$ in the special case  $k = n-1$.
Introducing a parameter $q$, we consider
the $q$-analogue 
\begin{equation}
B_{n,n-1}(X_n; q) \coloneqq \prod_{i = 1}^n (x_1 + \cdots + x_n + q x_i).
\end{equation}
This specializes to $B_{n,n-1}(x_1, \dots, x_n; q)$ at $q = -1$.
Switching to infinitely many variables,
we also consider the symmetric function
\begin{equation}
B_{n,n-1}(X;q) \coloneqq \sum_{j = 0}^n q^j \cdot e_j(X) \cdot h_{(1^{n-j})}(X).
\end{equation}

\subsection{The specialization $q = 0$}
At $q = 0$, we have the representation theoretic interpretation
\begin{equation}
B_{n,n-1}(X_n; 0) = h_1(X_n)^n = \ch(\overbrace{\CC^n \otimes \cdots \otimes \CC^n}^n),
\end{equation}
where $GL_n$ acts diagonally on the tensor product.
The symmetric function $B_{n,n-1}(X;0) = h_{(1^n)}(X)$ is the Frobenius image $\Frob(\CC[\symm_n])$ of the regular
representation of $\symm_n$.

\subsection{The specialization $q = -1$}
The case $q = -1$ is more interesting.  The symmetric function
\begin{equation}
B_{n,n-1}(X;-1) = \sum_{j = 0}^n (-1)^j \cdot e_j(X) \cdot h_{(1^{n-j})}(X)
\end{equation}
was introduced under the name $D_n$ by D\'esarm\'enien and Wachs
\cite{DW} in their study of derangements in the symmetric group.
Reiner and Webb \cite{RW} described the Schur expansion of
$B_{n,n-1}(X;-1)$ in terms of ascents in tableaux. Recall that an {\em
  ascent} in a standard Young tableau $T$ with $n$ boxes is an index
$1 \leq i \leq n-1$ such that $i$ appears in a row weakly below $i+1$
in $T$.  Athanasiadis generalized the Reiner-Webb theorem in the
context of the $\symm_n$ representation on the homology of the poset
of injective words \cite{Athanasiadis}.

\begin{theorem}
  (Reiner-Webb \cite[Prop. 2.3]{RW}) For $n \geq 2$ we have
  $B_{n,n-1}(X;-1) = \sum_{\lambda \vdash n} a_{\lambda} s_{\lambda}$,
  where $a_{\lambda}$ is the number of standard tableaux of shape
  $\lambda$ with smallest ascent given by an even number.  Here we
  artificially consider $n$ to be an ascent so every tableau has at
  least one ascent.
\end{theorem}

Gessel and Reutenauer discovered \cite[Thm. 3.6]{GR}
a relationship between $B_{n,n-1}(X_n; -1) = B_{n,n-1}(X_n)$
and free Lie algebras. Specifically, they proved
\begin{equation}
B_{n,n-1}(X_n; -1) = \sum_{\lambda} \ch(\mathrm{Lie}_{\lambda}(\CC^n)).
\end{equation}
The sum ranges over all partitions $\lambda \vdash n$ which have no
parts of size 1 and $\mathrm{Lie}_{\lambda}(\CC^n)$ is a
$GL_n$-representation called a {\em higher Lie module}
\cite{Reutenauer}.  Equivalently, if one expands $B_{n,n-1}(X_n; -1)$
into the basis of fundamental quasisymmetric functions, we have
\begin{equation}
B_{n,n-1}(X_n; -1) = \sum_{w \in D_n} F_{D(w)}
\end{equation}
where $D_n$ is the set of derangements in $\symm_n$ and
$D(w)=\{i: w(i)>w(i+1)\}$ is the descent set of $w$.

\subsection{The specialization $q = 1$}
At $q = 1$, the symmetric function $B_{n,n-1}(X;1)$ has an interpretation involving positroids.
A {\em positroid} of size $n$ is a 
length $n$ sequence $v_1 v_2 \dots v_n$ consisting of $j$ copies of the letter 0 (for some $0 \leq j \leq n$)
and one copy each of the  letters $1, 2, \dots, n-j$.
Let $P_n$ denote the set of positroids of size $n$.
For example, we have
\begin{equation*}
P_3 = \{ 123, 213, 132, 231, 312, 321, 012, 021, 102, 201, 120, 210, 001, 010, 100, 000 \}.
\end{equation*}
If we use a parameter $j$ to keep track of the number of  0s, we get
\begin{equation}
 |P_n| = \sum_{j = 0}^n \frac{n!}{j!}.
\end{equation}

A more common definition of positroids is permutations in $\symm_n$ with each fixed point colored white 
or black.
More explicitly, the 0's in $v = v_1 \dots v_n \in P_n$ correspond to the white fixed points and the remaining entries
of $v_1 \dots v_n$ are order-isomorphic to a unique permutation of the set $\{ 1 \leq i \leq n \,:\, v_i \neq 0 \}$;
the fixed points of this smaller permutation are colored black.
For example, if $v = 3020041 \in P_7$ then
\begin{equation*}
3020041 \leftrightarrow 6 2 3 4 5 7 1,  \quad \text{with white fixed points $2, 4, 5$ and black fixed point $3$}.
\end{equation*}
Positroids arise as an indexing set for Postnikov's cellular structure on the totally positive Grassmannian
\cite{Postnikov}.

Let $\CC[P_n]$ be the vector space of formal $\CC$-linear combinations of elements of $P_n$.
The symmetric group $\symm_n$ acts on $\CC[P_n]$ as follows. Let $1 \leq i \leq n-1$ and let 
$s_i = (i,i+1) \in \symm_n$ be the associated adjacent transposition.
If $v = v_1 \dots v_n \in P_n$, then we have
$s_i.v \coloneqq \pm v_1 \dots v_{i+1} v_i \dots v_n$ where the sign is $-$ if $v_i = v_{i+1} = 0$ and $+$ otherwise.
As an example, when $n = 4$ we have
\begin{equation*}
s_1.(2100) = 1200, \quad s_2.(2100) = 2010, \quad s_3.(2100) = -2100.
\end{equation*}
It can be checked that this rule satisfies the {\em braid relations}
\begin{equation}
\begin{cases}
s_i^2 = 1 & 1 \leq i \leq n-1 \\
s_i s_j = s_j s_i & |i - j| > 1 \\
s_i s_{i+1} s_i = s_{i+1} s_i s_{i+1} & 1 \leq i \leq n-2
\end{cases}
\end{equation}
and so extends to give an action of $\symm_n$ on $\CC[P_n]$.

\begin{proposition}
\label{positroid-proposition}
We have $\Frob( \CC[P_n] ) = B_{n,n-1}(X;1) = \sum_{j = 0}^n  e_j(X) \cdot h_{(1^{n-j})}(X)$.
\end{proposition}

\begin{proof}
For $0 \leq j \leq n$, let $P_{n,j} \subseteq P_n$ be the family of positroids with $j$  copies of 0.
Since the action of $\symm_n$ on $P_n$ does not change the number of 0s, 
the vector space direct sum 
$\CC[P_n] \cong \bigoplus_{j = 0}^n \CC[P_{n,j}]$ is stable under the action of $\symm_n$.
Since
$\Frob(\CC[P_n]) = \sum_{j = 0}^n \Frob(\CC[P_{n,j}])$, it is enough to check that
$\Frob(\CC[P_{n,j}]) = e_j(X) \cdot h_{(1^{n-j})}(X)$.

By our choice of signs in the action of $\symm_n$ on $\CC[P_{n,j}]$ and the definition of induction
product, we see that 
\begin{equation}
\CC[P_{n,j}] \cong \sign_{\symm_j} \circ \CC[\symm_{n-j}]
\end{equation}
where $\sign_{\symm_j}$ is the $1$-dimensional sign representation 
$\symm_j$ and $\CC[\symm_{n-j}]$ is the regular representation of $\symm_{n-j}$, so 
that 
\begin{equation}
\Frob(\CC[P_{n,j}]) = \Frob(\sign_{\symm_j}) \cdot \Frob(\CC[\symm_{n-j}]) = e_j(X) \cdot h_{(1^{n-j})}(X),
\end{equation}
as desired.
\end{proof}

We present a graded refinement of the module in Proposition~\ref{positroid-proposition}
in the next subsection.

\subsection{General $q$ and superspace quotients}  
We want to describe a graded $\symm_n$-module
whose graded Frobenius image equals $B_{n,n-1}(X;q)$.
This module will be a quotient of superspace.

For $n \geq 0$,
{\em superspace} 
 is the associative unital $\CC$-algebra with generators $x_1, \dots, x_n, \theta_1, \dots, \theta_n$
subject to the relations 
\begin{equation}
x_i x_j = x_j x_i, \quad x_i \theta_j = \theta_j x_i, \quad \theta_i \theta_j = - \theta_j \theta_i
\end{equation}
for all $1 \leq i, j \leq n$.
We write $\CC[x_1, \dots, x_n, \theta_1, \dots, \theta_n]$ for this algebra, with the understanding that 
the $x$-variables commute and the $\theta$-variables anticommute.
We can think of this as the ring of polynomial-valued differential forms on $\CC^n$.
In physics, the $x$-variables are called {\em bosonic} and the $\theta$-variables are called {\em fermionic}.
The symmetric group $\symm_n$ acts on superspace diagonally by the rule
\begin{equation}
w.x_i \coloneqq x_{w(i)}, \quad w.\theta_i \coloneqq \theta_{w(i)}, \quad w \in \symm_n, \, \, 1 \leq i \leq n.
\end{equation}

We define the {\em divergence free} quotient $DF_n$ of superspace by 
\begin{equation}
DF_n \coloneqq \CC[x_1, \dots, x_n, \theta_1, \dots, \theta_n]/ \langle x_1 \theta_1, x_2 \theta_2, \dots, x_n \theta_n \rangle.
\end{equation}
Here we think of superspace in terms of differential forms, so that
$x_i \theta_i$ is a typical contributor to the divergence of a vector
field.  The ideal defining $DF_n$ is $\symm_n$-stable and
bihomogeneous in the $x$-variables and the $\theta$-variables, so that
$DF_n$ is a bigraded $\symm_n$-module.  We use variables $t$ to keep
track of $x$-degree and $q$ to keep track of $\theta$-degree.

What is the bigraded Frobenius image $\grFrob(DF_n; t,q)$?  
For any subset $J = \{ j_1 < \cdots < j_k \} \subseteq [n]$, let 
$\theta_J \coloneqq \theta_{j_1} \cdots \theta_{j_k}$ be the corresponding product of $\theta$-variables in increasing
order.  Also let $\CC[X_J]$ be the polynomial ring over $\CC$ in the variables $\{ x_j \,:\, j \in J \}$ with indices
in $J$, so that $\CC[X_{[n]-J}]$ is the polynomial ring with variables whose indices do {\em not} lie in $J$. 
We have a vector space direct sum decomposition
\begin{equation}
\label{divergence-free-direct-sum}
DF_n = \bigoplus_{J \subseteq [n]} \CC[X_{[n]-J}] \cdot \theta_J.
\end{equation}

Let $DF_J \coloneqq \CC[X_{[n] - J}] \cdot \theta_J$ be the summand in
\eqref{divergence-free-direct-sum} corresponding to $J$.  The spaces
$DF_J$ are not closed under the action of $\symm_n$ unless
$J = \varnothing$ or $J = [n]$.  To fix this, for $0 \leq j \leq n$ we
set $DF_{n,j} \coloneqq \bigoplus_{|J| = j} DF_J$.  By
\eqref{divergence-free-direct-sum} we have
$DF_n = \bigoplus_{j = 0}^n DF_{n,j}$.  
Recall the plethystic formula for the graded Frobenius image of the polynomial ring:
\begin{equation}
\label{polynomial-graded-frobenius}
\grFrob(\CC[X_n];t) = h_n \left[ \frac{X}{1-t} \right].
\end{equation}
By the definition of induction
product and Equation~\eqref{polynomial-graded-frobenius} we have
\begin{align}
\grFrob(DF_{n,j}; q,t) &= q^j \cdot e_j(X) \cdot h_{n-j} \left[ \frac{X}{1-t} \right], \\
\grFrob(DF_{n}; q,t) &= \sum_{j = 0}^n q^j \cdot e_j(X) \cdot h_{n-j} \left[ \frac{X}{1-t} \right].
\end{align}

Let $I_n = \langle e_1(X_n), e_2(X_n), \dots, e_n(X_n) \rangle \subseteq DF_n$ be the ideal 
generated by the $n$ elementary symmetric polynomials
in the $x$-variables.  Equivalently, we can think of $I_n$ as the ideal generated by the vector space
$\CC[X_n]^{\symm_n}_+$ of symmetric polynomials with vanishing constant term within the divergence
free quotient of superspace.
Let $R_n \coloneqq DF_n/I_n$ be the corresponding bigraded quotient $\symm_n$-module.

\begin{theorem}
\label{bigraded-frobenius-identification}
The bigraded Frobenius image of $R_n$ is
\begin{equation}
\label{bigraded-frobenius-image}
\grFrob(R_n; q,t) = \sum_{j = 0}^n q^j \cdot e_j(X) \cdot 
\left[ \sum_{T \in \SYT(n-j)} t^{\maj(T)} \cdot s_{\shape(T)}(X) \right],
\end{equation}
where the sum is over all standard Young tableaux $T$ with $n-j$ boxes.
Consequently, we have
\begin{equation}
\grFrob(R_n; q,1) = \sum_{j = 0}^n q^j \cdot e_j(X) \cdot h_{(1^{n-j})}(X) = B_{n,n-1}(X;q).
\end{equation}
\end{theorem}

In the case $n = 3$, the standard tableaux with $\leq 3$ boxes are as follows:
\begin{equation*}
\varnothing \hspace{0.2in}
\begin{young} 1 \end{young} \hspace{0.2in}
\begin{young} 1 & 2 \end{young} \hspace{0.2in}
\begin{young} 1 \cr 2 \end{young} \hspace{0.2in}
\begin{young} 1 & 2 & 3 \end{young} \hspace{0.2in}
\begin{young} 1 & 2 \cr 3 \end{young} \hspace{0.2in}
\begin{young} 1 & 3 \cr 2 \end{young} \hspace{0.2in}
\begin{young} 1 \cr 2 \cr 3 \end{young}
\end{equation*}
From left to right, their major indices are $0, 0, 0, 1, 0, 2, 1, 3$. By 
Equation~\eqref{bigraded-frobenius-image},
\begin{equation*}
\grFrob(R_n;q,t) = q^3 e_3 + q^2 e_2 s_1 + q e_1 s_{2} + qt e_1 s_{11} 
+ s_3 + t s_{21} + t^2 s_{21} + t^3 s_{111}.
\end{equation*}

\begin{proof}
We can relate $DF_{\varnothing} = \CC[X_n]$ to the other $DF_J$ using maps. Specifically, define
\begin{equation}
\varphi_J : DF_{\varnothing} \rightarrow DF_J
\end{equation}
by the rule $\varphi_J(f(X_n)) = f(X_n) \theta_J$.  Each $\varphi_J$
is a map of $\CC[X_n]$-modules.

Let $I_n' \subseteq \CC[X_n]$ be the classical {\em invariant ideal}
$I_n' \coloneqq \langle e_1(X_n), e_2(X_n), \dots, e_n(X_n) \rangle$ which has the same
generating set as $I_n$, but is generated in the subring $\CC[X_n] \subseteq DF_n$. 
The ideals $I_n$ and $I_n'$ may be related as follows: for any subset $J \subseteq [n]$,
\begin{equation}
I_n \cap DF_J = \varphi_J(I_n')
\end{equation}
and furthermore there hold the graded vector space decompositions
\begin{equation}
I_n = \bigoplus_{J \subseteq [n]} (I_n \cap DF_J) = \bigoplus_{J \subseteq [n]} \varphi_J(I_n').
\end{equation}
Taking quotients, this gives
\begin{equation}
R_n = \bigoplus_{J \subseteq [n]} DF_J/(I_n \cap DF_J) = 
\bigoplus_{J \subseteq [n]} DF_J/\varphi_J(I_n').
\end{equation}

What does the quotient $DF_J/\varphi_J(I_n')$ look like?  Since $x_i  \theta_i = 0$ in
$DF_n$ for all $i$, for any $1 \leq k \leq n$ we have
\begin{equation}
\varphi_J(e_k(X_n)) = e_k(X_{[n] - J}) \cdot \theta_J,
\end{equation}
where $e_k(X_{[n] - J})$ is the degree $k$ elementary 
symmetric polynomial in the variable set indexed by $[n] - J$; observe that 
$e_k(X_{[n]-J}) = 0$ whenever $k > n-|J|$.
Since $\varphi_J$ is a map of $\CC[X_n]$-modules, the polynomials
$e_k(X_{[n] - J})$ for $0 < k \leq n-|J|$ form a generating set for the 
$\CC[X_n]$-module $\varphi_J(I_n')$.  Consequently, the map 
\begin{equation}
\CC[X_{[n]-J}]/\langle e_k(X_{[n] - J}) \,:\, 0 < k \leq [n]-J \rangle 
\xrightarrow{\, \, \, \, \,  \cdot \, \theta_J \, \, \,\, \,  }
DF_J/\varphi_J(I_n')
\end{equation}
given by multiplication by $\theta_J$ is a 
$\CC$-linear isomorphism which
maps a homogeneous polynomial of bidegree
$(q,t)$ to  $(q+j,t)$. 

The reasoning of the last paragraph shows that $R_n$ admits a direct sum decomposition
as a bigraded vector space:
\begin{equation}
R_n \cong 
\bigoplus_{J \subseteq [n]}
\CC[X_{[n]-J}]/\langle e_k(X_{[n] - J}) \,:\, 0 < k \leq n - |J| \rangle \otimes
\CC \{ \theta_J \},
\end{equation}
where $\CC \{ \theta_J \}$ is the $1$-dimensional $\CC$-vector space
spanned by $\theta_J$.
This may also be expressed with induction product as
a bigraded $\symm_n$-module isomorphism
\begin{equation}
R_n \cong \bigoplus_{j = 0}^n
\CC[X_{n-j}]/I'_{n-j} \circ
\CC \{ \theta_1 \theta_2 \cdots \theta_j \}.
\end{equation}
Since $\CC \{ \theta_1 \theta_2 \dots \theta_j \}$ carries the sign representation of $\symm_j$ in $q$-degree $j$
and $\CC[X_{n-j}]/I'_{n-j}$ is the classical coinvariant ring corresponding to $\symm_{n-j}$,
the claimed Frobenius image now follows from Theorem~\ref{thm:LusztigStanley}.
\end{proof}

\begin{remark} For $n > 0$,
Reiner and Webb \cite{RW} consider a chain complex  
\begin{equation*}
C_{\bullet} = ( \,  \cdots \rightarrow C_2 \rightarrow C_1 \rightarrow C_0 \rightarrow 0 \, )
\end{equation*}
whose 
degree $j$ component $C_j$ has basis given by length $j$ words $a_1 \dots a_j$
over the alphabet $[n]$ with no repeated letters. 
Up to a sign twist, the polynomial
\begin{equation}
\sum_{j = 0}^n q^j \cdot \Frob(C_j)
\end{equation}
coming from the graded action of $\symm_n$ on $C_{\bullet}$ (without taking homology)
equals $B_{n,n-1}(X;q)$.  At $q = 1$ (again up to a sign twist) we get the action of 
$\symm_n$ on the space $\CC[P_n]$ spanned by positroids of 
Proposition~\ref{positroid-proposition}.
\end{remark}

The key fact in the proof of Theorem~\ref{bigraded-frobenius-identification}
was that a $\CC[X_n]$-module generating set for $I_n \cap DF_J$
can be obtained by applying the map $\varphi_J$ to the generators of the $\CC[X_n]$-module
$I_n \cap DF_{\varnothing}$, or equivalently the generators of the ideal $I_n' \subseteq \CC[X_n]$.
Since the images of these generators under $\varphi_J$ have a nice form, it was possible 
to describe the $J$-component $DF_n/\varphi_J(I'_n) = DF_n/(I_n \cap DF_J)$ of $R_n$.

The program of the above paragraph can be carried out for a wider class of ideals.
For $r \leq k \leq n$, consider the ideal $I_{n,k,r} \subseteq DF_n$ with generators
\begin{equation}
I_{n,k,r} \coloneqq \langle x_1^k, x_2^k, \dots, x_n^k, e_n(X_n), e_{n-1}(X_n), \dots, e_{n-r+1}(X_n) \rangle.
\end{equation}
The ideal
\begin{equation}
I'_{n,k,r} \coloneqq I_{n,k,r} \cap DF_{\varnothing} = I_{n,k,r} \cap \CC[X_n]
\end{equation}
was defined by Haglund, Rhoades, and Shimozono and gives a variant of the coinvariant
ring whose properties are governed by ordered set partitions \cite{HRS}.
Pawlowski and Rhoades proved that the quotient of $\CC[X_n]$ by $I'_{n,k,r}$
and presents the cohomology of a certain variety of line configurations \cite{PR}
(denoted $X_{n,k,r}$ therein).

Let $R_{n,k,r} \coloneqq DF_n/I_{n,k,r}$ be the quotient of $DF_n$ by $I_{n,k,r}$. The same reasoning
as in the proof of Theorem~\ref{bigraded-frobenius-identification} gives
\begin{equation}
R_{n,k,r} \cong \bigoplus_{j = 0}^n
\CC[X_{n-j}]/I'_{n-j,k,r-j} \circ
\CC \{ \theta_1 \theta_2 \cdots \theta_j \},
\end{equation}
so that
\begin{equation}
\label{osp-frobenius}
\grFrob(R_{n,k,r};q,t) = \sum_{j = 0}^n q^j \cdot e_j(X) \cdot \grFrob(\CC[X_{n-j}]/I'_{n-j,k,r-j};t).
\end{equation}

The Schur expansion of the symmetric function \eqref{osp-frobenius} follows from 
material in \cite[Sec. 6]{HRS}.
Each term on the right-hand-side of Equation~\eqref{osp-frobenius} has the form
 $\grFrob(\CC[X_{n}]/I'_{n,k,r};t)$ for some $r \leq k \leq n$.
 Combining \cite[Lem. 6.10]{HRS} and \cite[Cor. 6.13]{HRS} we have
\begin{multline}
\grFrob(\CC[X_{n}]/I'_{n,k,r};t) = \\ \sum_{m = 0}^{k - r}  t^{m \cdot (n-k+m)} {k - r \brack m}_t
\left( \sum_{T \in \SYT(n)} t^{\maj(T)} {n - \des(T) - 1 \brack n-k+m}_t 
s_{\shape(T)}(X) \right).
\end{multline}
Here we adopt the $t$-binomial notation
\begin{equation}
{n \brack k}_t := \frac{[n]_t!}{[k]!_t [n-k]!_t}, \quad 
[n]!_t := [n]_t [n-1]_t \cdots [1]_t, \quad
[n]_t := 1 + t + \cdots + t^{n-1}.
\end{equation}

\section{Conclusion}
\label{Open}

As an extension of Problem~\ref{gl-module-problem}, one could ask for a module whose Weyl character 
is given by the expression in Pragacz's Theorem~\ref{pragacz-theorem}.
For simplicity, let us consider the case of one rank $n$ vector bundle $\EEE$ with Chern roots $x_1, \dots, x_n$
and the Chern plethysm $s_{\lambda}(\SSS^{\mu}(\EEE))$ for two partitions $\lambda$ and $\mu$.
If $W$ is a $GL_n$-module with Weyl character $s_{\lambda}(\SSS^{\mu}(\EEE))$, then 
\begin{equation}
\dim W = s_{\lambda}(\SSS^{\mu}(\EEE)) \mid_{x_1 = \cdots = x_n = 1} =
|\mu|^{|\lambda|} \cdot |\SSYT( \lambda, \leq |\SSYT(\mu, \leq n)|) |,
\end{equation}
where the second equality uses the fact that $\SSS^{\mu}(\EEE)$ has Chern roots 
$\sum_{\square \in \mu} x_{T(\square)}$ where $T$ ranges over all elements of $\SSYT(\mu, \leq n)$.

The quantity $|\SSYT( \lambda, \leq |\SSYT(\mu, \leq n)|) |$ has a natural representation theoretic interpretation
via Schur functor composition:
\begin{equation}
\dim \SSS^{\lambda}(\SSS^{\mu}(\CC^n)) = |\SSYT( \lambda, \leq |\SSYT(\mu, \leq n)|) |.
\end{equation}
This suggests the following problem.

\begin{problem}
\label{general-pragacz-problem}
Let $W$ be the vector space 
$\mathrm{Hom}_{\CC}( \SSS^{\lambda}(\SSS^{\mu}(\CC^n)), (\CC^{|\mu|})^{\otimes |\lambda|})$.
Find an action of $GL_n$ on $W$ whose Weyl character equals $s_{\lambda}(\SSS^{\mu}(\EEE))$.
\end{problem}

The natural $GL_n$-action on $W$ coming from acting on $\CC^n$ does not 
solve Problem~\ref{general-pragacz-problem}. Indeed, this is a polynomial
representation of $GL_n$ of degree $|\lambda| \cdot |\mu|$
whereas the polynomial $s_{\lambda}(\SSS^{\mu}(\EEE))$ has degree $|\lambda|$.
We remark that the Weyl character of the $GL_n$-action on $\SSS^{\lambda}(\SSS^{\mu}(\CC^n))$
coming from the action of $GL_n$ on $\CC^n$ is the classical plethysm
$s_{\lambda}[s_{\mu}]$.
Problem~\ref{general-pragacz-problem} asks for the corresponding representation theoretic 
operation for Chern plethysm.


We close with some connections between Boolean product polynomials 
and the theory of maximal unbalanced collections.
Let $\ee_1, \ee_2, \dots, \ee_n$ be the standard basis of $\RR^n$ and
for $S \subseteq [n]$, let $\ee_S := \sum_{i \in S} \ee_i$ be the sum of the coordinate
vectors in $S$.
A collection of subsets $\CCC \subseteq 2^{[n]}$ is called {\em
    balanced} if the convex hull of the vectors $\ee_S$ for $S \in \CCC$
  meets the main diagonal
  $\{ (t, t, \dots, t) \,:\, 0 \leq t \leq 1 \}$ in
  $[0,1]^n$. Otherwise, the collection $\CCC$ is {\em
    unbalanced}. 
    
   Balanced collections were defined by Shapley
  \cite{Shapley} in his study of $n$-person cooperative games.  In the
  containment partial order on $2^{[n]}$, balanced collections form an
  order filter and unbalanced collections form an order ideal, so we
  can consider {\em minimal balanced} and {\em maximal unbalanced}
  collections.  Minimal balanced collections were considered by
  Shapley \cite{Shapley} and maximal unbalanced collections arose
  independently in the work of Billera-Moore-Moraites-Wang-Williams
  \cite{BMMWW} and Bj\"orner \cite{Bjorner}.  In particular, Billera
  et. al. gave a bijection between maximal unbalanced collections and
  the regions of the resonance arrangement \cite{BMMWW}.  Thus,
  counting maximal unbalanced collections is equivalent to counting
  the chambers in the resonance arrangement defined by the polynomial $B_n(X_n)$.
  
  One way to count the chambers of the resonance arrangement would be to 
  find the characteristic polynomial of the matroid $M_n = \{ \ee_S \,:\, \varnothing \neq S \subseteq [n] \}$.
  To understand the matroid $M_n$, one would need to know whether the determinant $\det A$ 
  of any $0,1$-matrix $A$ of size $n \times n$ is zero or not.
  Determinants of $0,1$-matrices arise  \cite{wiki:Hadamard} in {\em Hadamard's maximal determinant problem},
  which asks whether there exists an $n \times n$ $0,1$-matrix $A$ such that 
   $\det A = (n+1)^{(n+1)/2}/2^{n}$ (the inequality $\leq$ is known to hold for any matrix $A$).
   The study of the matroid $M_n$ could shed light on the Hadamard problem.

\section{Acknowledgments}
\label{Acknowledgements}

This project started at BIRS at the Algebraic Combinatorics workshop
in August 2015; the authors thank Lou Billera for the initial
inspiration for pursuing this project at that workshop and the
collaboration leading to \cite{BBT}.  The authors are also grateful to
Christos Athanasiadis, Patricia Hersh, Steve Mitchell, Jair Taylor,
and Alex Woo for helpful conversations.  We would also like to thank
Matja\v{z} Konvalinka and Philippe Nadeau for enlightening discussions
regarding the proof of Theorem~\ref{lascoux-schur-expansion}.
S. Billey was partially supported by NSF Grants DMS-1101017 and
DMS-1764012.  B. Rhoades was partially supported by NSF Grant
DMS-1500838.

\end{document}